\documentclass[11pt, a4paper]{article}
\usepackage{amssymb,amsthm,amsmath,amscd}     
\usepackage{latexsym}                         
\usepackage{enumerate}

\usepackage{verbatim}

\usepackage{color}
\usepackage{url}

\newtheorem{thm}{Theorem}[section]
\newtheorem{cor}[thm]{Corollary}
\newtheorem{lem}[thm]{Lemma}

\theoremstyle{definition}
\newtheorem{defn}[thm]{Definition}
\newtheorem{rem}[thm]{Remark}
\newtheorem{exa}[thm]{Example}

\newcommand{\N}{\ensuremath{\mathbb{N}}}
\newcommand{\Z}{\ensuremath{\mathbb{Z}}}

\newcommand{\bfx}{\ensuremath{\mathbf x}}
\newcommand{\bfy}{\ensuremath{\mathbf y}}

\DeclareMathOperator{\mydig}{dig}   
\DeclareMathOperator{\myint}{int}   
\DeclareMathOperator{\XOR}{XOR}   

\allowdisplaybreaks[1]

\begin{document}

\title{Research Note: Adding Digit Vectors}

\author{Peter Hellekalek\thanks{Research supported by the
University of Salzburg, projects P1884/5-2009 and P1884/4-2010.} }

\date{\today}

\maketitle

\renewcommand{\thefootnote}{}
\renewcommand{\thefootnote}{\arabic{footnote}}
\setcounter{footnote}{0}

\begin{abstract}
In this paper, we study the different possibilities to add two vectors of digits of a given length $m$.
Our results show that there are at least $2^{m-1}$ different additions of such vectors,
while there exist only two {\em types} of addition that we may employ, addition with carry and addition without carry.
The proofs of our results are elementary.
\end{abstract}

\section{Introduction}\label{intro}
In this research note we investigate the different possibilities to add two digit vectors of the same length.

Addition of digit vectors, in particular addition of binary vectors,
is employed in many algorithms.
Prominent examples in applied cryptography are the block ciphers IDEA~\cite{Lai91a} and AES~\cite{Daemen02a}
and several stream ciphers.
We refer the reader to \cite{Menezes97a} and to the eSTREAM-project\footnote{\url{http://www.ecrypt.eu.org/stream/}}
for  details on such ciphers.

The author's starting point into this question was the following observation.
Any construction method for finite or infinite sequences of points 
is based on some arithmetical operations like addition or multiplication,
on a suitable domain.
It is most helpful if the algebraic structure underlying these operations is an abelian
group.
The choice of this group determines which function systems will be suitable
for the analysis of a given sequence,
because the construction method is intrinsically related to
function systems,
via the concept of the {\em dual group} (see Hewitt and Ross~\cite{Hewitt79a}).
Different types of
sequences require different types of function systems for their analysis.
An example of such a suitable ``match'' between sequences and function systems
in the theory of uniform distribution of sequences in the $s$-dimensional unit cube $[0,1)^s$ is given
by Kronecker sequences or,
in their discrete version, good lattice points,
and the trigonometric functions.
This construction method is based on
{\em addition modulo one}
(see Niederreiter~\cite[Ch. 5]{Nie92a} and Sloan and Joe~\cite{Slo94a}).
A second example are digital  nets and sequences and,
as appropriate function system, the Walsh functions.
Here,
{\em addition without carry} of digit vectors comes into play
(see Niederreiter  \cite[Ch. 4]{Nie92a} and Dick and Pillichshammer~\cite{Dick10a}).

One important type of a digital sequence,
the Halton sequence,
can also be generated by {\em addition with carry},
the underlying group being the compact abelian group of $b$-adic integers.
From the search for the appropriate function system in this case,
the notion of the $b$-adic function system originated.
This concept was developed in series of papers 
(see \cite{Hel09a,Hel10a,Hel10c,Hel11a,Hel11b}, and for a background in ergodic theory \cite{Hel12b}).

These investigations led to the question if there are any other types of
addition of digit vectors, because if not, then the Walsh functions in base $b$ and the $b$-adic
function system and their combination in a {\em hybrid function system} (see \cite{Hel10c} for this notion)
cover all possible cases of function systems associated with additions of digit vectors.

Our results below show that this is indeed the case:
there are {\em only two types} of addition of digit vectors:
addition without carry, which is also called $\XOR$-addition,
and addition with carry, which is also known as integer addition.

We exhibit that,
for a given length $m$ of the digit vectors with digits in some given integer base $b\ge 2$,
there are at least $2^{m-1}$ different additions for such vectors.
This large number may be increased considerably
if we employ also automorphisms of suitable groups of residues.

Our reasoning is elementary.
It is based on a classical theorem on finite abelian groups and on the notion of compositions of positive integers.

The ideas presented below might have applications in cryptography,
for example in stream or block cipher algorithms.
If the information which digits are added in which way in the enciphering scheme is kept secret,
then this will add not only to {\em confusion},
but, as already used in IDEA, changing the type of addition also adds to {\em diffusion}
(for these two notions, see \cite{Shannon49a}).
Hence,
breaking the cipher would be more difficult.

\section{Addition of digit vectors}

Let $b\ge 2$ be a fixed integer and let
$\mathcal{A}_b= \{ 0,1, \ldots, b-1\}$ denote the set of $b$-ary digits.
For $m\in \N$,
let $\mathcal{A}_b^m$ stand for
the $m$-fold cartesian product of the set $\mathcal{A}_b$ with itself.

We will study the following question,
mostly in the case $b=p$ a prime:
What are the  binary operations ``+'' on the set $\mathcal{A}_b^m$ of  digit vectors such that the pair
$(\mathcal{A}_b^m, +)$ is an abelian group?

\begin{rem}
In this paper, when we speak of an ``addition on $\mathcal{A}_b^m$'',
we mean a binary operation ``+'' on the set $\mathcal{A}_b^m$ of  digit vectors in base $b$ such that the pair
$(\mathcal{A}_b^m, +)$ is an abelian group.
\end{rem}

The reader should note that the term ``binary'' has two different meanings in this paper,
which will become clear from the context.
A binary operation on a set $G$ is a map from the cartesian product $G\times G$ into $G$.
Referring to the representation of real numbers in base $b=2$,
the elements of the set $\mathcal{A}_2^m$ are called binary vectors,
and for $m=1$ one speaks of binary digits.

Let us consider the case $b=2$ first.
There are two well known examples for addition of digit vectors.
One is addition {\em without} carry and the other is addition {\em with} carry.

For $n\in \N$, $n\ge 2$,
let  $\Z/n\Z$ denote the additive group of residue classes modulo $n$.
We identify this cyclic group with the set of integers $\{0,1, \ldots, n-1\}$ equipped with addition modulo $n$.

\begin{exa}[Addition without carry]
We identify $\mathcal{A}_2$ with $\Z/2\Z$.
For $\bfx, \bfy \in \mathcal{A}_2^m$,
$\bfx = (x_0, \ldots, x_{m-1})$ and $\bfy = (y_0, \ldots, y_{m-1})$,
we define
\[
\bfx + \bfy = (x_0\oplus y_0, \ldots, x_{m-1}\oplus y_{m-1}),
\]
where `$\oplus$' denotes addition on $\Z/2\Z$,
$0\oplus 0=1\oplus 1 =0$, and $0\oplus 1=1\oplus 0 =1$.
The pair $(\mathcal{A}_2^m, +)$ is an abelian group.
In fact,
it is isomorphic to the product group $(\Z/2\Z)^m$.
We call this binary operation {\em addition without carry},
or  {\em $\XOR$-addition} of digit vectors.
\end{exa}

Any nonnegative integer $k$, $0\le k < 2^m$,
has a unique dyadic representation of the form $k=k_0+k_1 2 + \dots +k_{m-1} 2^{m-1}$
with digits $k_j\in \mathcal{A}_2$, $0\le j\le m-1$.

\begin{exa}[Addition with carry]
We identify $\mathcal{A}_2^m$ with the group $\Z/2^m\Z$.
For $\bfx\in \mathcal{A}_2^m$,
$\bfx = (x_0, \ldots, x_{m-1})$ ,
we define the map $\myint_2: \mathcal{A}_2^m \rightarrow \Z/2^m \Z$,
\[
\myint_2(\bfx) = x_0+ x_1 2+\dots + x_{m-1}2^{m-1}.
\]
Further, let $\mydig_2: \Z/2^m \Z \rightarrow \mathcal{A}_2^m$,
\[
\mydig_2(k) = (k_0,  k_1, \ldots , k_{m-1}),
\]
where $k=k_0 + k_1 2 + \dots + k_{m-1}2^{m-1}$ is the representation of $k$ in base 2.
Finally,
for $\bfx, \bfy \in \mathcal{A}_2^m$,
we define
\[
\bfx + \bfy = \mydig_2( \myint_2(\bfx) + \myint_2(\bfy) \pmod{2^m}).
\]

With this binary operation the pair $(\mathcal{A}_2^m, +)$ is an abelian group.
Clearly,
it is isomorphic to the additive group $\Z/2^m \Z$.
We call this type of binary operation {\em  addition with carry} or {\em integer addition} of digit vectors.
\end{exa}

For $m\ge 2$,
our two examples are non-isomorphic groups,
because one is cyclic and the other is not.

Apart from these two examples,
are there any other possibilities to define addition on the set $\mathcal{A}_2^m$
of binary digit vectors?

From the Fundamental Theorem for Finite Abelian Groups (see \cite[Sec. 10]{Herstein99a}) we obtain the following corollary.
In this context, a {\em partition} \index{partition} of a positive integer $m$ is a
finite sequence $(t_i)_{i=1}^r$, $r\in \N$, of positive integers
with the two properties
(i)  $t_1 \ge t_2 \ge  \dots \ge t_r$, and
(ii) $t_1+t_2+ \dots + t_r = m$.

\begin{cor}
\label{ch:intro:cor:FTFAG}
The non-isomorphic groups of order $2^m$, $m\in \N$, are given by
the product groups
\[
(\Z/2^{t_1}\Z) \times (\Z/2^{t_2}\Z) \times \dots \times (\Z/2^{t_r}\Z),
\]
where $(t_i)_{i=1}^r$ is a partition of $m$.
\end{cor}

Hence, in view of Corollary \ref{ch:intro:cor:FTFAG},
an addition on the set $\mathcal{A}_2^m$
is defined if we put
\begin{equation}\label{eqn:2-adic}
(\mathcal{A}_2^m, +) \cong (\Z/2^{t_1}\Z) \times (\Z/2^{t_2}\Z) \times \dots \times (\Z/2^{t_r}\Z),
\end{equation}
where $m=t_1+ t_2 +\dots + t_r$ is a partition of $m$.
Here,
the symbol ``$\cong$'' denotes that the two groups are isomorphic.

As a consequence,
there are at least as many possibilities to define addition
on the set $\mathcal{A}_2^m$ of binary digit vectors of length $m$,
as there are different partitions of the integer $m$.

From the structure of the factors in (\ref{eqn:2-adic}) we obtain the following information.
\begin{cor}\label{cor:types}
The only two types of binary operations on (sub)vectors of digits that may appear in the group
law of the abelian group $(\mathcal{A}_2^m, +)$ are the following:
\begin{itemize}
\item addition given by finite product groups of the form
$(\Z/2\Z)\times \dots \times (\Z/2\Z)$,
which is what we have called $\XOR$-addition, or
\item addition in groups of residues of the form $\Z/2^t\Z$, $t\ge 2$,
which we have called integer addition.
\end{itemize}
\end{cor}

Denote the number of different partitions of $m$ by $P(m)$.
We refer to the monograph \cite[Ch. 2.5.1]{HDCM00a} for details on the partition function $P$
like tables, or for results on its asymptotic behavior.

For example, if $m=8$,
then there are $P(8)=22$ non-isomorphic groups of order $2^{8}$,
like the groups $\Z/2^8\Z$, $(\Z/2^7\Z)\times (\Z/2\Z)$,
$(\Z/2^6\Z)\times (\Z/2^2\Z)$, and so on.
Among these 22 non-isomorphic groups of order $2^8$,
let us choose for illustration the group
\[
(\Z/2^{3}\Z) \times (\Z/2^{2}\Z)     \times (\Z/2\Z)^3 .
\]
What addition on $\mathcal{A}_2^8$ does this group induce?
Addition of two bytes $\bfx=(x_0,  \ldots, x_7)$ and
$\bfy=(y_0, \ldots, y_7)$ is carried out as follows.
The first three bits of $\bfx$ and $\bfy$ are interpreted as the three binary digits
of two integers in the range $\{0,1, \ldots, 7\}$.
These two integers are added, the resulting integer is reduced modulo $2^3$,
which gives an integer in the range from 0 to 7,
and the three binary digits of this integer give the first three digits of the sum
$\bfx + \bfy$.
In other words,
for the first three bits,
we carry out addition in the group $\Z/2^3\Z$ of residue classes modulo $2^3$.
The same procedure,
which we have called integer addition,
is applied to the next two bits.
For the last three bits, the digit vectors $(x_5,x_6,x_7)$ and $(y_5, y_6, y_7)$ are
$\XOR$-ed, because we have to perform addition in the group $(\Z/2\Z)^3$.

Observe that this is not the complete answer to our question.
Our question concerned
the different possibilities  to add two binary vectors of length $m$.
In the definition of such an addition,
the position of each digit matters,
whereas in
Corollary \ref{ch:intro:cor:FTFAG} the order of the factors does not.
The two groups
\[
(\Z/2^{3}\Z) \times (\Z/2^{2}\Z) \times (\Z/2\Z)^3
\]
and
\[
(\Z/2\Z)^3 \times (\Z/2^{2}\Z) \times (\Z/2^{3}\Z)
\]
are isomorphic,
but they induce different additions on $\mathcal{A}_2^8$.
In the first case, the first three digits of $\bfx+\bfy$ are computed via integer addition,
in the second case these three digits are computed by the $\XOR$-operation.

In terms of representing a number $m\in  \N$,
this means we consider the representation $8=3+2+1+1+1$ to be different from $8=1+1+1+2+3$,
because they induce different additions on $\mathcal{A}_2^8$.
This leads us to the following definition.

\begin{defn}
A {\em composition} of a positive integer $m$ is a
finite sequence of positive integers $(t_i)_{i=1}^r$, $r\in \N$,
with the  property $m=t_1+t_2+ \dots + t_r$.
\end{defn}

Two such sequences which differ in the order of their summands are deemed to be different compositions,
while they would be considered to be the same partition of $m$.
The following result and its nice proof are well known.

\begin{lem}
\label{lem:compositions}
Let $C(m)$ denote the number of different compositions of $m\in \N$.
Then
\[
C(m)= 2^{m-1}.
\]
\end{lem}

\begin{proof}
The case $m=1$ is trivial.
Let $m\ge 2$.
In the scheme
\[
1 \Box 1 \Box \ldots \Box 1 \Box 1,
\]
of $m$ 1's and $m-1$ boxes,
we may replace every box either by  a plus sign or by a comma.
A different choice for each of the $m-1$ boxes leads to a different composition of $m$.
\end{proof}

We may summarize our findings as follows.

\begin{thm}\label{thm:2-additions}
For the set  $\mathcal{A}_2$ of binary digits and for $m\in \N$,
the following holds for all binary operations ``+'' on the set $\mathcal{A}_2^m$ such that
the pair $(\mathcal{A}_2^m, +)$ forms an abelian group:

\begin{enumerate}
\item There are only two types of addition of (sub)vectors of binary digits,
addition without carry and addition with carry.
\item The number of different additions 
on $\mathcal{A}_2^m$
that arise  from the compositions of $m$
is equal to $2^{m-1}$.
\end{enumerate}
\end{thm}

The preceeding arguments may be generalized directly to the case of an arbitrary prime base $p$ instead of
base 2.

\begin{cor}
Let $p$ be a prime.
Then there exist only two types of addition for vectors of $p$-ary digits,
addition without carry, which corresponds to addition on finite product groups of the form $(\Z/p\Z)\times \dots \times(\Z/p\Z)$,
and addition with carry, which corresponds to groups of the form $\Z/p^t\Z$, $t\ge 2$.
Further,
for every $m\in \N$,
the number of additions
on $\mathcal{A}_p^m$ that arise from the compositions of $m$
is equal to $2^{m-1}$.
\end{cor}

In case of a composite integer base $b$,
$b\ge 2$,
the situation is more complicated,
because from the factorization of $b$ into distinct prime powers many `small' cyclic groups arise in
the Fundamental Theorem for Finite Abelian Groups that cannot directly be related to operations on the $b$-adic digits,
as we did above.
One will have to use the Chinese Remainder Theorem to treat these cases.
We omit these technical details because they do not contribute any new aspects to our investigation.

Clearly,
every  composition of the positive integer $m$ defines an addition on $\mathcal{A}_b^m$,
by simply following the recipes given above.
Hence, the number of different binary operations ``+'' on $\mathcal{A}_b^m$
such that the pair $(\mathcal{A}_b^m, +)$ is a abelian group
 is at least $2^{m-1}$.

In detail,
the composition $m=t_1 + \dots + t_r$
gives rise to the addition on $\mathcal{A}_b^m$ defined by the following product group:
\begin{equation}\label{eqn:b-adic}
(\mathcal{A}_b^m, +) \cong (\Z/b^{t_1}\Z) \times (\Z/b^{t_2}\Z) \times \dots \times (\Z/b^{t_r}\Z).
\end{equation}

\section{Combination with automorphisms}

In cryptographic applications of the ideas above,
for example in stream ciphers,
the information which of the $m$ digits are added by $\XOR$-addition and which by integer addition
might become part of the key in the encryption scheme.

We will increase the key space considerably by the following idea.
Suppose that we have chosen the composition $m=t_1+\dots + t_r$ of $m$.
Hence,
we obtain the group law on $\mathcal{A}_b^m$ from the product group given in  (\ref{eqn:b-adic}).
For a given factor $\Z/b^t \Z$ of this product
we may combine  integer addition with an arbitrary automorphism $\sigma$ of the group  $\Z/b^t \Z$ as follows.
For $\bfx, \bfy\in \mathcal{A}_b^t$, $\bfx=(x_0, \ldots, x_{t-1})$ and $\bfy=(y_0, \ldots, y_{t-1})$,
these $t$ digits are added by the law
\begin{equation}\label{eqn:sigma-addition}
\bfx +\bfy = \mydig_b\left( \sigma(\myint_b(\bfx)) + \sigma(\myint_b(\bfy)) \pmod{b^t}\right).
\end{equation}

\begin{lem}\label{lem:automorphisms}
There are $\varphi(b^t)$ different ways to define the addition in (\ref{eqn:sigma-addition}).
\end{lem}

\begin{proof}
The following reasoning is standard.
Let $\sigma$ be an homomorphism of the additive group $\Z/b^t \Z$ into itself.
Then $\sigma(a) = a \sigma(1)$ for all $a\in \Z/b^t \Z$.
Hence, $\sigma$ is an automorphism if and only if $(a, b^t)=1$.
In other words, $a$ belongs to the (multiplicative) group of prime residues $(\Z/b^t \Z)^*$ modulo $b^t$,
which has $\varphi(b^t)$ elements.
\end{proof}

\begin{thm}\label{thm:p-adic}
Let $p$ be a prime.
Then, for any $m\in \N$,
the compositions of $m$ and the automorphisms of the associated  groups of residues
generate
\[
(p-1) (2p-1)^{m-1}
\]
different additions on $\mathcal{A}_p^m$,
i.e., binary operations ``+''  such that the pair
$(\mathcal{A}_p^m, +)$ is an abelian group.
\end{thm}

\begin{proof}
For a given composition $m=t_1+\dots+t_r$ of $m$ into $r$ components,
$1\le r\le m$,
we obtain the group law from
\begin{equation}\label{eqn:p-adic-composition}
(\mathcal{A}_p^m, +) \cong (\Z/p^{t_1}\Z) \times (\Z/p^{t_2}\Z) \times \dots \times (\Z/p^{t_r}\Z).
\end{equation}
Due to Lemma \ref{lem:automorphisms},
this group product allows
\[
p^m \left( \frac{p-1}{p} \right)^r
\]
automorphisms.

The number of compositions of $m$ with $r$ terms is equal to the number of possibilities to
chose $r-1$ from the available $m-1$
places to put a comma in the proof of Lemma \ref{lem:compositions}.
For this reason, there exist
\[m-1\choose r-1
\]
compositions of $m$ into $r$ terms.

Hence,
the total number of different additions on $\mathcal{A}_p^m$
that stem from compositions of $m$ and the associated automorphisms is given by
\[
p^m \sum_{r=1}^m {m-1 \choose r-1} \left( \frac{p-1}{p}\right)^r = (p-1) (2p-1)^{m-1}.
\]
\end{proof}

\begin{exa}
Let $p=2$ and $m=8$.
There are $2^7=128$ different additions on $\mathcal{A}_2^8$
arising from the 128 compositions of the number 8.

For the given composition $8=3 + 4 + 1$,
there are $\varphi(2^3)= 4$ different integer additions for the first three bits and
$\varphi(2^4)=8$ for the next four bits,
if we employ the combination of addition with automorphisms like in (\ref{eqn:sigma-addition}),
and there is just one addition for the last bit.
As a consequence,
for this particular composition of $m=8$,
there exists not only one addition of 8-bit dyadic vectors but
there are 32 different additions available due to the combination with the 4 automorphisms of the factor
$\Z/2^3\Z$ and the 8 automorphisms of $\Z/2^4\Z$.

Hence,
if we allow automorphisms of the residue groups that appear as factors in the product group (\ref{eqn:p-adic-composition}),
then from Theorem \ref{thm:p-adic} is follows that there are $3^7=2187$ different additions of 8-bit vectors available.
\end{exa}

In the case of an arbitrary integer base $b$,
the result is the following:

\begin{thm}\label{thm:b-adic}
Let $b\ge 2$ be an integer.
Then, for any $m\in \N$,
the compositions of $m$ and the automorphisms of the associated  groups of residues
generate
\[
b^m C_b (1+C_b)^{m-1}
\]
different binary operations ``+'' on $\mathcal{A}_b^m$ such that the pair
$(\mathcal{A}_b^m, +)$ is an abelian group.
Here, the number $C_b$ is defined as
\[
C_b = \prod_{i=1}^s (1-1/p_i),
\]
where $b= \prod_{i=1}^s p_i^{\alpha_i}$ is the factorization of $b$ into distinct primes $p_i$,
with $\alpha_i\in \N$, $1\le i\le s$.
\end{thm}

\begin{proof}
We translate the proof of Theorem \ref{thm:p-adic} step by step from the case of a prime base $p$ to the general base $b$.
\end{proof}

\subsection*{Acknowledgements}
The author would like to thank Markus Neuhauser (RWTH Aachen) for a most helpful remark
and the NUHAG research group at the University of Vienna for its hospitality during a stay
at the Faculty of Mathematics that proved to be most fruitful.


%

\def\cdprime{$''$} \def\cdprime{$''$} \def\cdprime{$''$}

\begin{small}
\noindent\textbf{Author's address:}\\
Peter Hellekalek,
Fachbereich Mathematik, Universit\"{a}t Salzburg, Hellbrunnerstr. 34, 5020 Salzburg, Austria\\
E-mail: \texttt{peter.hellekalek@sbg.ac.at}
\end{small}

\end{document}